\documentclass[12pt]{amsart}

\usepackage{amsmath,amssymb,amsthm}
\usepackage[dvipdfm]{graphicx}
\usepackage[dvips]{color}

\newtheorem{Thm}{\textbf{Theorem}}[section]
\newtheorem{Lem}[Thm]{\textbf{Lemma}}
\newtheorem{Cor}[Thm]{\textbf{Corollary}}
\newtheorem{Prop}[Thm]{\textbf{Proposition}}

\theoremstyle{remark}
\newtheorem{Rem}[Thm]{\rm{Remark}}

\theoremstyle{definition}
\newtheorem{Def}[Thm]{\rm{Definition}}

\newtheorem*{ack}{\rm{Acknowledgment}}
\newtheorem*{conv}{\rm{Convention}}

\newcommand{\Spec}{\mathop{\mathrm{Spec}}\nolimits}
\newcommand{\Shpec}{\mathop{\mathit{Spec}}\nolimits}
\newcommand{\codim}{\mathop{\mathrm{codim}}\nolimits}

\newcommand{\Phroj}{\mathop{\mathit{Proj}}\nolimits}

\newcommand{\mindiscrep}{\mathop{\mathrm{mindiscrep}}\nolimits}

\newcommand{\DF}{\mathop{\mathrm{DF}}\nolimits}

\newcommand{\cent}{\mathop{\mathrm{center}}\nolimits}

\newcommand{\cond}{\mathop{\mathrm{cond}}\nolimits}

\begin{document}

\title{A generalization of \linebreak 
Ross-Thomas' slope theory} 

\author{Yuji Odaka}
\dedicatory{Dedicated to Professor Toshiki Mabuchi on his sixtieth birthday} 
\subjclass[2010]{Primary 14L24; Secondary 14J17, 32Q15. }
\keywords{Donaldson-Futaki invariants, K-stability, semi-log-canonicity. }
\address{Research Institute for Mathematical Sciences (RIMS), 
Kyoto University, Oiwake-cho, Kitashirakawa, Sakyo-ku, Kyoto
606-8502, Japan}
\email{yodaka@kurims.kyoto-u.ac.jp}


\begin{abstract}
We give a formula of the Donaldson-Futaki invariants for certain type of semi test configurations, which essentially generalizes Ross-Thomas' slope theory \cite{RT07}. 
The positivity (resp.\ non-negativity) of those ``a priori special" Donaldson-Futaki invariants implies K-stability (resp.\ K-semistability). We show its applicability by proving K-(semi)stability of certain polarized varieties with semi-log-canonical singularities, generalizing some results of \cite{RT07}. 

\end{abstract}

\maketitle


\section{Introduction}

The GIT stability has been introduced in the aim of 
specifying the objects to be parametrized in quasi-projective moduli schemes 
by Mumford \cite{Mum65}. The objects 
which we study here \color{black}{are}\color{black}{} polarized varieties. That theme is recently put much attention as the relation with the problem of existence of 
``canonical" K\"ahler metrics come into play. Along that development, the \textit{K-stability} is formulated as a newer kind of the GIT stability by Tian \cite{Tia97} and reformulated by Donaldson \cite{Don02}, which is conjectured to be an algebro-geometric paraphrase of the existence of a K\"ahler metric with constant scalar curvature (\textit{cscK} metric,  in short). In this paper, we provide some basic results towards concrete 
solution for the general problem ``When a polarized variety is GIT-stable?" as a first in 
a series (cf.\ \cite{Od09}, \cite{Od10}, \cite{OS10}) as their foundation. 
Mainly, we treat K-stability. 

The K-stability is defined as 
the positivity of the \textit{Donaldson-Futaki invariants} (also called as the  \textit{generalized Futaki invariants}). 
Roughly speaking, they are a kind of GIT weights 
associated to the \textit{test configurations}, which can be regarded as 
the ``geometrization" of $1$-parameter subgroups from the GIT viewpoint. 
From the viewpoint of differential geometry, the Donaldson-Futaki invariant 
generalizes the Futaki's obstruction \cite{Fut83} 
to the existence of K\"ahler-Einstein metric on a Fano manifold in algebro-geometric way. More precisely, it generalizes a value of the Futaki characters \cite{Fut83} at a generator of $\mathbb{C}^{*}$-action on a Fano manifold, which should vanish if there is a K\"ahler-Einstein metric on it. 

Recently, Ross introduced the concept of \textit{slope stability} as 
an analogue of the original slope stability for vector 
bundles by Mumford and Takemoto, and systematically studied with Thomas \cite{RT06}, 
\cite{RT07}. Let $(X,L)$ be a polarized variety which we are interested. 
Then, essentially that theory is an explicit description of 
the Donaldson-Futaki invariants of some special test configurations, 
a blow up of a closed subscheme of $X\times \mathbb{A}^{1}$ which is \textit{scheme-theoritically} supported in $X\times \{0\}$, which is coined as 
\textit{the deformation to the normal cone} by Fulton. 
The \textit{slope stability} is defined as those positivity. 
Therefore, K-stability implies slope stability. 
As its applications, many examples most of which are even 
smooth are proved to be unstable. 
However, it is not enough in the aim of studying K-stability, in the sense that the $2$ points blow up of projective plane is later proven to be slope stable by Panov and Ross \cite{PR09} but it is known to be K-unstable. Please consult \cite{RT06}, \cite{RT07}, 
\cite{PR09} for their theory. 

In this paper, we generalize their theory by treating the test configurations of the form 
of the blow up of more general ideals 
(\textit{flag ideals}) of $X\times \mathbb{A}^{1}$, 
and give an explicit formula \ref{DF.formula} of the Donaldson-Futaki invariants of those. 

The formula \ref{DF.formula} is applicable in two senses. 
Firstly, it is general enough in the sense that the positivity  (resp.\ non-negativity) of those ``a priori special" Donaldson-Futaki invariants implies K-stability (resp.\ K-semistability) as we will see in Corollary \ref{usefulness}. 

Secondly, those Donaldson-Futaki invariants are described in a analyzable form 
as a sum of two parts, 
the \textit{canonical divosor part}, which reflects the global ``positivity" of the canonical divisor, and the \textit{discrepancy term}, which reflects the singularities. 
Please consult Theorem \ref{DF.formula} for the detail of our formula. 

As simplest applications, 
we provide algebro-geometric straightforward proofs of 
K-semistability of Calabi-Yau varieties and K-stability of curves, admitting some mild singularities. 

\begin{Cor}[=Theorem \ref{easy.Kstab}]\label{i.easy.Kstab}

{\rm (i)}
A semi-log-canonical canonically polarized curve $(X,L=\omega_{X})$ is K-stable. 

{\rm (ii)}
A semi-log-canonical polarized variety $(X,L)$ with numerically trivial $K_{X}$ 
is K-semistable. 

\end{Cor}
\noindent
The notion of semi-log-canonical singularities, forming a class of mild singularities, 
were first introduced by Koll\'{a}r and Shepherd-Barron 
\cite{KSB88} for $2$-dimensional case and extended by Alexeev \cite{Ale96} for higher dimensional case. It is defined in terms of \textit{discrepancy}, which is developed along the log minimal model program as a fundamental invariant of singularities. A variety is  simply called semi-log-canonical if it has only semi-log-canonical singularities. For the details, consult the original paper \cite{Ale96} 
and the textbook \cite[section 2.3 and section 5.4]{KM98} on the basics of discrepancy. 




We should remark that, thanks to the recent works on Yau's conjecture; 
the analogue of Kobayashi-Hitchin correspondence 
\cite{Don05}, \cite{CT08}, \cite{Stp09}, \cite{Mab08b}, \cite{Mab09} and the affirmative solution to the Calabi conjecture \cite{Yau78}, a differential geometric proof of Theorem \ref{i.easy.Kstab} is known for the case $X$ is smooth over $\mathbb{C}$. Also an 
algebro-geometric proof of (\rm{i}) is known to \cite[Corollary 6.7]{RT07} for smooth case and the slope stability for (\rm{ii}) was proved for the case with at worst 
canonical singularities by an algebro-geometric method in \cite[Theorem 8.4]{RT07}. 


We should also note that, after having written the first draft of this paper, 
the author noticed that a similar formula of Donaldson-Futaki invariants had already been discovered by Professor 
X.~Wang \cite[Proposition19]{Wan08}. The differences with our formula is 
essentially twofolds. Firstly, we 
extends the setting to semi test configurations of which we take advantage, 
under the style of the formula as a generalization of Ross-Thomas' slope theory. 
Secondly, the proofs are quite different as Wang's proof 
depends on his beautiful relation between GIT weights and \textit{heights}  \cite[Theorem8]{Wan08}, while ours depends on an old Lemma of \cite{Mum77}. 

Please consult \cite{Od09}, \cite{Od10} and \cite{OS10} for more applications of 
the formula \ref{DF.formula} as sequels. 


This paper is organized as follows. 
In the next section, we will review the basic stability notions for polarized varieties. For the readers' 
convenience, we include Mabuchi's proof \cite{Mab08a} of the equivalence of asymptotic Hilbert stability and asymptotic 
Chow stability in a simplified but essentially the same form. 
In section \ref{formula}, we will introduce the key formula \ref{DF.formula} of Donaldson-Futaki invariants and show that K-stability (resp.\ K-semistability) follows from only 
those positivity (resp.\ non-negativity). In section \ref{app}, we give the applications. 

\begin{conv}

We work over an algebraically closed field $k$ with characteristic zero, 
unless otherwise stated. An algebraic scheme means a finite type and separated scheme over $k$. A variety means a reduced algebraic scheme. 

A projective scheme means a complete (algebraic) scheme which 
has some ample invertible sheaves. $(X,L)$ always denotes a polarized scheme, a projective scheme $X$ with a polarization $L$, which means an ample invertible sheaf. Furthermore, we always assume $X$ to be reduced,  equidimensional, 
satisfies Serre condition $S_2$ and Gorenstein in codimension $1$ for simplicity. 

For singularities, for a divisor $e$ over a normal variety $X$ (cf.\ \cite{KM98}), $a(e;X)$ denotes the discrepancy of $e$ under the assumption of $\mathbb{Q}$-Gorensteiness of  $X$ and 
$a(e;(X,D))$ denotes the discrepancy of $e$ on a log pair $(X,D)$ 
(i.\ e.\ a pair of a normal variety $X$ and its Weil divisor $D$ with $\mathbb{Q}$-Cartier 
$K_{X}+D$ ). These notation about discrepancy follows those of \cite[section 2.3]{KM98}, 
which we refer to for the details. 

\end{conv}

\begin{ack}

First of all, the author would like to express sincere gratitude to his advisor Professor Shigefumi Mori for his warm encouragements, suggestions and reading the drafts. The author also would like to thank Professors 
Shigeru Mukai, Noboru Nakayama, Masayuki Kawakita for useful suggestions, especially throughout the seminars in the master terms and Mr.\ Kento Fujita for providing him the name for ``canonical divisor part". 

He also wants to thank Professors Julius Ross and Xiaowei Wang very much 
for inspiring communications, preparing nice environment during the author's visits. 

He appreciates the comments on the draft by Professor Yuji Sano and Professor  Yongnam Lee. 
Finally, the special thanks go to Professor Toshiki Mabuchi 
for answering his questions for several times and encouragements. 

The author is supported by the Grant-in-Aid for Scientific Research (KAKENHI No.\ 21-3748) and the 
Grant-in-Aid for JSPS fellows. 

\end{ack}


\section{The stability notions} 

In this section, we will review the basic of the stability notions for polarized varieties. 
There are a few of well known versions: 
K-stability, asymptotic Chow stability, asymptotic Hilbert stability and their semistable versions. 
Originally, Gieseker \cite{Gie82} introduced the asymptotic Hilbert stability which was confirmed for canonically 
polarized surfaces with at worst canonical singularities. 
Asymptotic Chow stability was introduced in \cite{Mum77} 
and K-stability was introduced firstly by Tian in \cite{Tia97}, and extended 
and reformulated by Donaldson \cite{Don02}. The motivation for 
introducing the K-(semi, poly)stability is to 
seek the GIT-counterpart of 
the existence of special K\"{a}hler metric, 
as an analogy of the Kobayashi-Hitchin correspondence for vector bundles. Let us recall that ``$*$
-unstable'' means that ``not $*$-\textit{semi}stable'' . 

At first, we review the definition of asymptotic stabilities. 

\begin{Def}

A polarized scheme $(X,L)$ is said to be \textit{asymptotically Chow stable} (resp.\ \textit{asymptotically Hilbert stable}, 
\textit{asymptotically Chow semistable}, 
\textit{asymptotically Hilbert semistable}), if for an arbitrary 
$m \gg 0$, $\phi _{m}(X)\subset \mathbb{P}(H^{0}(X, L^{\otimes{m}}))$ is Chow stable (resp.\ Hilbert stable, Chow semistable, 
Hilbert semistable), where $\phi _{m}$ 
is the closed immersion defined by the complete linear system $|L^{\otimes{m}}|$. 
 
\end{Def} 

To define the K-stability, we review the concept of test configuration following Donaldson \cite{Don02}. 
Our notation (and even expression) almost follows \cite{RT07}, so we refer to it for details. 

\begin{Def}

A \textit{test configuration} (resp.\ \textit{semi test configuration}) for a polarized scheme $(X,L)$ is a 
polarized scheme $(\mathcal{X},\mathcal{L})$ with: 
\begin{enumerate}
\item{a $\mathbb{G}_{m}$ action on $(\mathcal{X},\mathcal{L})$}
\item{a proper flat morphism $\alpha\colon \mathcal{X} \rightarrow \mathbb{A}^{1}$}
\end{enumerate}
such that $\alpha$ is $\mathbb{G}_{m}$-equivariant for the usual action on $\mathbb{A}^{1}$: 
\begin{align*}
\mathbb{G}_{m}\times \mathbb{A}^{1}&& \longrightarrow&& \mathbb{A}^{1}\\
                          (t,x)    && \longmapsto    &&    tx,      
\end{align*}
$\mathcal{L}$ is relatively ample (resp.\ relatively semi ample), 
and $(\mathcal{X},\mathcal{L})|_{\alpha^{-1}(\mathbb{A}^{1}\setminus\{0\})}$ is $\mathbb{G}_{m}$-equivariantly isomorphic 
to $(X,L^{\otimes r})\times (\mathbb{A}^{1}\setminus\{0\})$ for some positive integer $r$, called \textit{exponent}, 
with the natural action of $\mathbb{G}_{m}$ on the latter and the trivial action on the former. 

\end{Def}

\begin{Prop}[{\cite[Proposition 3.7]{RT07}}]\label{tc.1-ps}

In the above situation, a one-parameter subgroup of $GL(H^{0}(X,L^{\otimes{r}}))$ is equivalent to the data of a test configuration 
\color{black}{$(\mathcal{X},\mathcal{L})$ whose polarization $\mathcal{L}$ is very ample  (over $\mathbb{A}^{1}$) }\color{black}{}
with exponent $r$ of $(X,L)$ for $r \gg 0$. 

\end{Prop}

We will call the test confinguration which corresponds to a one parameter 
subgroup, called the \textit{DeConcini-Procesi family}. 
(Its curve case appears in \cite[Chapter $4$ $\S6$]{Mum65}.) 
Therefore, the test configuration can be regarded as \textit{geometrization} of one-parameter subgroup. 
This is a quite essential point for our study, as in Ross and Thomas' slope theory \cite{RT06}, \cite{RT07}. 

The \textit{total weight} of an action of $\mathbb{G}_{m}$ on some finite-dimensional vector space is 
defined as the sum of all weights. Here the \textit{weights} mean the exponents of eigenvalues which should be powers of $t$. 
We denote the total weight of the induced action on $(\alpha_{*}\mathcal{L}^{\otimes{K}})|_{0}$ as 
$w(Kr)$ and $\dim X$ as $n$. 
It is a polynomial of $K$ of degree $n+1$. 
We write $P(k):=\dim H^{0}(X,L^{\otimes{k}})$. Let us take $rP(r)$-th power and 
SL-normalize the action of $\mathbb{G}_{m}$ on 
$(\alpha_{*}\mathcal{L})|_{0}$, then the corresponding normalized weight on 
$(\alpha_{*}\mathcal{L}^{\otimes{K}})|_{0}$ 
is $\tilde{w}_{r,Kr}:=w(k)rP(r)-w(r)kP(k)$, where $k:=Kr$. It is a 
polynomial of form 
$\sum_{i=0}^{n+1}e_{i}(r)k^{i}$ of degree $n+1$ in $k$ for $k \gg 0$, with coefficients which are also 
polynomial 
of degree $n+1$ in $r$ 
for $r \gg 0$ : $e_{i}(r)=\sum_{j=0}^{n+1}e_{i,j}r^{j}$ for $r \gg 0$. Since the weight is normalized, 
$e_{n+1,n+1}=0$. $e_{n+1,n}$ is 
called the \textit{Donaldson-Futaki invariant} of the test configuration, which we will denote as 
$\DF(\mathcal{X},\mathcal{L})$. 
Let us recall that $(n+1)!e_{n+1}(r)r^{n+1}$ is the Chow weight of $X\subset \mathbb{P}(H^{0}(X,L^{\otimes r}))$ 
\cite[Lemma 2.11]{Mum77}. For an arbitrary \textit{semi} test configuration $(\mathcal{X},\mathcal{L})$ 
of exponent $r$ (cf. \cite{RT07}), we can also define the (normalized) Chow weight or 
the Donaldson-Futaki invariant as well by setting $w(Kr)$ as the totalweight of the induced action on $H^{0}(\mathcal{X}, 
\mathcal{L}^{\otimes K})/tH^{0}(\mathcal{X}, \mathcal{L}^{\otimes K})$. 

\begin{Def}

A polarized scheme $(X,L)$ is \textit{K-stable} 
(resp.\ \textit{K-semistable}, \textit{K-polystable}) 
if for all $r \gg 0$, for any nontrivial test configuration 
for $(X,L)$ with exponent $r$ the leading coefficient $e_{n+1,n}$ of $e_{n+1}(r)$ 
(the Donaldson-Futaki invariant) is 
positive (resp.\ non-negative, positive if $\mathcal{X} \not\cong X\times \mathbb{A}^{1}$ and 
nonnegative otherwise).  

\end{Def}

We should note that the original K-stability of \cite{Don02} is what is called K-\textit{poly}stability in \cite
{RT07}. We follow the convention of \cite{RT07}. 
These are related as follows. 

Asymptotically Chow stable $\Rightarrow$ Asymptotically Hilbert stable 
$\Rightarrow$ Asymptotically Hilbert semistable $\Rightarrow$ Asymptotically Chow semistable 
$\Rightarrow$ K-semistable. 
 
It is easy to prove the above, so we omit the proof (see \cite{Mum77}, \cite{RT07}). 
We end this section by proving the equivalence of two asymptotic stability notions, 
following the paper \cite{Mab08a} but in a more simplified form, for readers' convenience. We should note that 
its semistability version is not proved anywhere in literatures, as far as the author knows. 

\begin{Thm}[{\cite[Main Theorem (b)]{Mab08a}}]

For a polarized scheme over an arbitrary algebraically closed field, 
asymptotic Hilbert stability and asymptotic Chow stability are equivalent. 

\end{Thm}

\begin{proof}

We prove this along the idea of \cite{Mab08a}. 
The formulation is a little different, 
but essentially the same. 
We make full use of the framework of test configuration. This proof 
is valid over an arbitrary algebraically closed field with any characteristic. 

Let us recall the basic criterion of asymptotic stabilities as in 
\cite[Theorem 3.9]{RT07}. 
$(X,L)$ is asymptotically Chow stable (resp.\ asymptotically Hilbert stable) 
if and only if for all $r \gg 0$, any nontrivial test 
configuration for $(X,L)$ with exponent $r$ has $e_{n+1}(r)>0$ (resp.\ $\tilde{w}_{r,k}>0$ for all $k \gg 0$). 
Therefore, asymptotic Chow stability implies asymptotic Hilbert stability. 
(Actually, Chow stability implies Hilbert stability as well). 
To prove the converse, we assume that $\tilde{w}_{r,k}>0$ for all $k \gg r \gg 0$. 

Since 
\begin{equation*}
\Biggl(\dfrac{\tilde{w}_{r,kk'}}{kk'P(kk')}\Biggr)-\Biggl(\dfrac{\tilde{w}_{r,k}}{kP(k)}\Biggr) 
=\Biggl(\dfrac{rP(r)}{k^{2}k'P(kk')P(k)}\Biggr)\times \tilde{w}_{k,kk'}  
\end{equation*}
and $\tilde{w}_{k,kk'}$ is positive by our assumption, 
the inequality $\dfrac{\tilde{w}_{r,kk'}}{kk'P(kk')}>\dfrac{\tilde{w}_{r,k}}{kP(k)}$ holds for 
all $k' \gg k \gg r \gg 0$. Therefore, we can take a monotonely-increasing sequence $k_{i} (i=0,1,\dots)$ divisible by $r$, 
and $k_{0}=r$ with $\dfrac{\tilde{w}_{r,k_{i}}}{k_{i}P(k_{i})}$ increasing. 
$\dfrac{\tilde{w}_{r,k_{i}}}{k_{i}P(k_{i})}$ converges since the denominator is a polynomial of $k_{i}$ 
of degree $n+1$ and the numerator is a polynomial of $k_{i}$ of degree at most $n+1$. 
In our case, the initial term 
is $\dfrac{\tilde{w}_{r,k_{0}}}{k_{0}P(k_{0})}=0$, so 
the sequence converges to a positive number, which should have 
the same sign as $e_{n+1}(r)$. This completes the proof. 

\end{proof}


\section{A formula of Donaldson-Futaki invariants}\label{formula} 

In this section, we prove the main formula of the Donaldson-Futaki invariants of (certain type of) semi test configurations, and establish some results on the semi test configurations which assure the usefulness of the formula. 
As we noted in the introduction, a same type formula of Donaldson-Futaki invariants had already been given independently 
for a test configuration with a (relatively) \textit{ample} polarization by Professor X.~Wang \cite{Wan08}, earlier than us. 
The differences are essentially twofolds, as we explained in the introduction. 
Firstly, we define the class of ideals, which we use for our study of stability. 

\begin{Def}

Let $(X,L)$ be an $n$-dimensional polarized variety. 
A coherent ideal $\mathcal{J}$ of $X\times \mathbb{A}^{1}$ is called a \textit{flag ideal} if 
$\mathcal{J}=I_{0}+I_{1}t+\dots+I_{N-1}t^{N-1}+(t^{N})$,  
where $I_{0}\subseteq I_{1}\subseteq \dots I_{N-1} \subseteq \mathcal{O}_{X}$ is the sequence of coherent ideals. 
(It is equivalent to that the ideal is $\mathbb{G}_{m}$-invariant under the natural action 
of $\mathbb{G}_{m}$ on $X\times \mathbb{A}^{1}$.) 
\end{Def}

Let us introduce some notation. 
We set $\mathcal{L}:=p_{1}^{*}L$ on $X\times \mathbb{P}^{1}$ or $X\times \mathbb{A}^{1}$, 
and denote the $i$-th projection morphism from $X \times \mathbb{A}^{1}$ or $X \times \mathbb{P}^{1}$ by $p_{i}$. 
Let us write the blowing up as $\Pi \colon \mathcal{B}(:=Bl_{\mathcal{J}}(X\times \mathbb{P}^{1}))\rightarrow X\times 
\mathbb{P}^{1}$ and the natural exceptional Cartier divisor as $E$, i.e.\ $\mathcal{O}(-E)=\Pi^{-1}\mathcal{J}$. 
Let us assume $\mathcal{L}^{\otimes r}(-E)$ is (relatively) semi-ample (over $\mathbb{A}^{1}$) and 
 consider the Donaldson-Futaki invariant of the blowing up (semi) test configuration 
$(\mathcal{B}, \mathcal{L}^{\otimes r}(-E))$. Now, we can state our main formula.  

\begin{Thm}\label{DF.formula} 

Let $(X,L)$ and $\mathcal{B}$, $\mathcal{J}$ be as above. And we assume that exponent $r=1$. 
$($It is just to make the formula easier. For general $r$, put $L^{\otimes r}$ and $\mathcal{L}^{\otimes r}$ 
to the place of $L$ and $\mathcal{L}$. $)$ 
Furthermore, we assume that $\mathcal{B}$ is Gorenstein in codimension $1$. 
Then the corresponding Donaldson-Futaki invariant $\DF((Bl_{\mathcal{J}}(X\times \mathbb{A}^{1}), \mathcal{L}(-E)))$  is 

\begin{equation*}
\dfrac{1}{2(n!)((n+1)!)}\bigl\{-n(L^{n-1}.K_{X})(\mathcal{L}(-E))^{n+1}+(n+1)(L^{n})
((\mathcal{L}(-E))^{n}.\Pi^{*}(p_{1}^{*}K_{X}))
\end{equation*}
\begin{equation*}
+(n+1)(L^{n})((\mathcal{L}(-E))^{n}.K_{\mathcal{B}/X\times \mathbb{A}^{1}})\bigr\}. 
\end{equation*}
In the above, all the intersection numbers are taken on $X$ or $\bar{\mathcal{B}}:=Bl_{\mathcal{J}}(X\times \mathbb{P}^{1})$, 
which are complete schemes. 

\end{Thm}
We call the sum of first two terms the \textit{canonical divisor part} 
since they involve intersection numbers with the canonical divisor $K_{X}$ or its pullback, 
and the last term the \textit{discrepancy term} since it reflects discrepancies over $X$. 
This division into two parts plays an important role in our applications 
(cf.\ section \ref{app}, \cite{Od10}, \cite{OS10}). 
\begin{proof}

By definition, the Donaldson-Futaki invariant is the coefficient of $k^{n+1}r^{n}$ in $w(k)rP(r)-w(r)kP(k)$ under the same 
notation as in the previous section. Therefore, it is enough to calculate $w(k)$ modulo $O(k^{n-1})$. 

Firstly, we interpret the weight $w(k)$ as a dimension of a certain vector space, 
through the following lemma \cite[Lemma(2.14)]{Mum77} which was called ``droll Lemma" by Mumford. 

\begin{Lem}[{\cite[Lemma(2.14)]{Mum77}}]\label{dr.lem} 
Let $V$ be a vector space over $k$ and assume that $\mathbb{G}_{m}$ acts on $V\otimes_{k}k[t]$, where $V$ is a vector space 
over $k$, by acting $V$ trivially and $t$ by weight $(-1)$. For a sequence of subspaces of $V$, $V_{0}\subseteq V_{1} \subseteq \cdots \subseteq 
V_{N-1} \subseteq V_{N}= \cdots =V$, let us set $\mathcal{V}:=\sum {V_{i}}t^{i}$ which is a sub $k[t]$ module of $V\otimes_{k}k[t]$. 
Then, the total weight on $\mathcal{V}/t\mathcal{V}$ is equal to $-\dim(V\otimes _{k}k[t] / \mathcal{V})$. 
\end{Lem} 

From this lemma, it follows that 
$$w(k)=-\dim(H^{0}(X\times\mathbb{A}^{1},\mathcal{L}^{\otimes{k}})/H^{0}(X\times\mathbb{A}^{1},\mathcal{J}^{k}\mathcal{L}^{\otimes{k}})). $$ 

\begin{Lem}\label{high.coh}

$h^{i}(X\times\mathbb{A}^{1},\mathcal{J}^{k}\mathcal{L}^{\otimes{k}})=O(k^{n-1})$ for $i>0$. 

\end{Lem}

\begin{proof}[Proof of {\em{Lemma \ref{high.coh}}}]

By our assumption, $\mathcal{L}(-E)$ is (relatively) semiample (over $\mathbb{A}^{1}$).  Therefore, its global section (the direct image sheaf of the projection 
onto $\mathbb{A}^{1}$) and $\mathcal{L}^{\otimes{k_{0}}}(-k_{0}E)$ for large enough $k_{0}$ induces a morphism $f\colon 
\mathcal{B}\rightarrow \mathcal{C}$, which is isomorphic over $\mathbb{A}\setminus 
\{0\}$. Let $\mathcal{M}$ be the canonical ample invertible sheaf with $f^{*}\mathcal{M}=\mathcal{L}^{k_{0}}(-k_{0}E)$. 
Since $H^{i}(X\times \mathbb{A}^{1},\mathcal{J}^{kk_{0}}\mathcal{L}^{\otimes{kk_{0}}})=H^{i}(\mathcal{B},\mathcal{L}^{\otimes{kk_{0}}}(-kk_{0}E))
$ $=H^{0}(\mathcal{C},(R^{i}f_{*}\mathcal{O}_{\mathcal{B}})\otimes \mathcal{M}^{\otimes{k}})$ and we have the support of 
$R^{i}f_{*}\mathcal{O}_{\mathcal{B}}$ only on the image of $f$-exceptional set 
\color{black}{(i.e., the locus in $\mathcal{C}$ where $f$ is not finite) }\color{black}{}
whose dimension is less than or equal to $(n-1)$, the lemma holds. 

\end{proof}

Using Lemma \ref{high.coh}, we can see that for $k \gg 0$; 

\begin{equation*}
 \begin{split}
  &  -\dim(H^{0}(X\times\mathbb{A}^{1},\mathcal{L}^{\otimes{k}})/H^{0}(X\times\mathbb{A}^{1},\mathcal{J}^{k}\mathcal{L}^{\otimes{k}})) \\
  &=-h^{0}(\mathcal{L}^{\otimes{k}}/\mathcal{J}^{k}\mathcal{L}^{\otimes{k}})+O(k^{n-1}) \\
  &=\chi (X\times \mathbb{P}^{1},\mathcal{J}^{k}\mathcal{L}^{\otimes{k}})-\chi (X\times \mathbb{P}^{1},\mathcal{L}^{\otimes{k}})+O(k^{n-1}).  
 \end{split}                                  
\end{equation*}

Finally, using \color{black}{the}\color{black}{} weak Riemann-Roch formula of the following type, we obtain the formula by \color{black}{simple calculation}\color{black}{}, 
which we omit here. 

\begin{Lem}[Weak Riemann-Roch formula]
For an $n$-dimensional polarized variety $(X,L)$ which is Gorenstein in codimension $1$, 

$$\chi (X,L^{\otimes{k}})=\dfrac{(L^{n})}{n!}k^{n}-\dfrac{(L^{n-1}.K_{X})}{2((n-1)!)}k^{n-1}+O(k^{n-2}),$$ 
where $(L^{n-1}.K_{X})$ is well-defined since $X$ is Gorenstein in codimension $1$. 
\end{Lem}

\end{proof}

\begin{Rem}

The formula \ref{DF.formula} can also be deduced from the formula of Chow weight by Mumford \cite[Theorem(2.9)]{Mum77}, 
as we did (implicitly) in \cite{Od09}. As Mumford obtained it by using the \textit{droll Lemma} \ref{dr.lem}, these proofs are essentially the same. 
 
\end{Rem}

From now on, we will argue to show the usefulness of our formula \ref{DF.formula} 
(cf.\ Corollary \ref{usefulness}). Let us continue fixing a polarized variety $(X,L)$ 
and think of its semi test configurations. 
We prepare the following notion. 

\begin{Def}\label{pn}
A semi test configuration $(\mathcal{X}, \mathcal{L})$ is \textit{partially normal} 
if any prime divisor supported on the singular locus of $\mathcal{X}$ 
projects surjectively onto $\mathbb{A}^{1}$. 
\end{Def}
\noindent
For example, a normal semi test configuration is partially normal of course. 
This notion is intended to extend the normality of semi test configuration 
for not necessarily normal $X$.

\begin{Prop}\label{partial normalization}
For an arbitrary test configuration $(\mathcal{X},\mathcal{L})$, 
there exists a finite surjective birational morphism $f \colon \mathcal{Y}\rightarrow \mathcal{X}$, where $(\mathcal{Y}, f^{*}\mathcal{L})$ is a partially-normal test configuration, with $\DF(\mathcal{Y}, f^{*}\mathcal{L})\leq\DF(\mathcal{X},\mathcal{L})$. 
\end{Prop}

\begin{proof}
If $X$ is normal, we can simply take the normalization of the test configuration. 
Even if $X$ is not normal, and $\mathcal{X}$ is not partially-normal, we can still 
``partially normalize" $\mathcal{X}$ as follows. 

Let us take the normalization $\nu\colon \mathcal{X}^{\nu}\rightarrow \mathcal{X}$ and take $p\nu\colon 
(\mathcal{Y}:=)\Shpec_{\mathcal{O}_{\mathcal{X}}}(i_{*}\mathcal{O}_{X\times (\mathbb{A}\setminus\{0\})}\cap \mathcal{O}_{\mathcal{X}^{\nu}})
\rightarrow 
\mathcal{X}$, where $i\colon X\times(\mathbb{A}^{1}\setminus\{0\})\rightarrow 
X\times \mathbb{A}^{1}$ is the open immersion. Obviously, $p\nu$ is finite as a morphism. We call this $\mathcal{Y}$ as the \textit{partial normalization} of the 
semi test configuration $\mathcal{X}$. 

This partial normalization is partially-normal as a test configuration 
(Definition \ref{pn}) due to the following Lemma. 
\begin{Lem}\label{dummy}
The morphism $\mathcal{X}^{\nu}\rightarrow \mathcal{Y}$ is an isomorphism over an open neighborhood of the generic points of the central fiber. 
\end{Lem}

\begin{proof}
Let us take an open affine subscheme $U(\cong \Spec R)\subset \mathcal{X}$ which includes all the generic points of the central fiber in $\mathcal{X}$. 
Then the preimage of $U$ in $\mathcal{Y}$ is $\Spec (R[t^{-1}]\cap R^{\nu})$. If we 
take small enough $U$, $R[t^{-1}]$ is normal so that $R^{\nu}\subset R[t^{-1}]$. 
This completes the proof. 
\end{proof}
\noindent
The normalization or the partial normalization $\mathcal{Y}$ of semi test configuration has the canonical $\mathbb{G}_{m}$ 
-linearized polazation, the pullback of the linearized polarization of the original test configuration. 

Then, $\DF(\mathcal{Y},f^{*}\mathcal{L})\leq\DF(\mathcal{X},\mathcal{L})$ by 
\cite[Proposition 5.1]{RT07}, whose claim holds and the proof essentially works without  the normality condition of $X$. 
\end{proof}

\begin{Prop}\label{toblup}

For an arbitrary partially normal test configuration $(\mathcal{X},\mathcal{M})$, 
there is a flag ideal $\mathcal{J}$ and $r,s\in \mathbb{Z}_{>0}$ such that its blow up 
$(\mathcal{B}:=Bl_{\mathcal{J}}(X\times \mathbb{A}^{1}), \mathcal{L}^{\otimes r}(-E))$  is a semi test configuration, which is Gorenstein in codimension $1$, dominating $(\mathcal{X},\mathcal{M}^{\otimes s})$ by a morphism 
$f\colon \mathcal{B}\rightarrow \mathcal{X}$ such that $\mathcal{L}^{\otimes r}(-E)=f^{*}\mathcal{M}$ and $\DF(\mathcal{B},\mathcal{L}^{\otimes r}(-E))=\DF(\mathcal{X},\mathcal{M}^{\otimes s})$.  
\end{Prop}

\begin{proof}
Firstly, we take a $\mathbb{G}_{m}$-equivariant resolution of the indeterminancy 
of a natural birational map $h \colon X\times \mathbb{A}^{1}\dashrightarrow \mathcal{X}$ as follows. Since the indeterminancy locus $Z$ of $h$ has codimension 
at least $2$ in $X\times \mathbb{A}^{1}$, if we write $j \colon (X\times \mathbb{A}^{1})\setminus Z  
\color{black}{ \hookrightarrow}\color{black}{} 
 X\times \mathbb{A}^{1}$ the natural open immersion, 
then $j_{*}h^{*}\mathcal{M}^{\otimes s}$ for $s\in \mathbb{Z}_{>0}$ 
is canonically isomorphic to $\mathcal{L}^{\otimes r}$ for some $r\in \mathbb{Z}_{>0}$, 
by the Serre's $S_2$ property of $X\times \mathbb{A}^{1}$ which follows from the $S_2$ condition of $X$, which is assumed in the ``Convention". 
If we take sufficiently large $s$, then $\mathcal{M}^{\otimes s}$ 
is (relatively) very ample over $\mathbb{A}^{1}$ and so 
$h$ is defined by the relative linear system over $\mathbb{A}^{1}$. 
\color{black}{Take a basis of $H^{0}(\mathcal{X},\mathcal{M}^{\otimes s})$ 
as a free $k[t]$-module, which consists of eigenvectors of the naturally associated  $\mathbb{G}_{m}$-action. 
They induces sections of $h|_{((X\times \mathbb{A}^{1})\setminus Z)}^{*}\mathcal{M}^{\otimes s}$ and so, they also define global sections of $\mathcal{L}^{\otimes r}$ because 
$Z$ has codimension at least $2$, as we noted. }\color{black}{}
Therefore, there is a flag ideal $\mathcal{J}'$ where 
those global sections of $\mathcal{L}^{\otimes r}$ generate the subsheaf  $\mathcal{J}'\mathcal{L}^{\otimes r}\subset \mathcal{L}^{\otimes r}$. 
\color{black}{We note that $\mathcal{O}/\mathcal{J}'$ is not necessarily 
supported in $Z$. }\color{black}{}
If we blow up the flag ideal $\mathcal{J}'$, we obtain a resolution of indeterminancy of $h$. Let us write it as $\mathcal{B}':=Bl_{\mathcal{J}'}(X\times \mathbb{A}^{1})\rightarrow \mathcal{X}$ and let $E'$ be the exceptional Cartier divisor with $\mathcal{O}_{\mathcal{B}'}(-E')=\mathcal{J}'\mathcal{O}_{\mathcal{B}'}$. 

Furthermore, we can take the partial normalization $\mathcal{B}$ of $\mathcal{B}'$ as before. By the arguments of Lemma \ref{dummy}, $\mathcal{B}$ is Gorenstein in codimension $1$. Let us write the projection $\mathcal{B}\rightarrow X\times\mathbb{A}^{1}$ as $\Pi$. 
Then, if we put $\mathcal{J}:=\Pi_{*}(p\nu)^{*}\mathcal{O}_{\mathcal{B}'}(-mE')$ for 
sufficiently large $m\in \mathbb{Z}_{>0}$, it is a flag ideal whose blow up is $\mathcal{B}$  itself. Furthermore, if we write $f$ the morphism from $\mathcal{B}$ to $\mathcal{X}$,  $f^{*}\mathcal{M}^{\otimes s}=\mathcal{L}^{\otimes r}(-E)$ where $E=(p\nu)^{*}E'$. 

\color{black}{
We want to prove $\DF(\mathcal{B},\mathcal{L}^{\otimes r}(-E))=\DF(\mathcal{X},\mathcal{M}^{\otimes s})$. For that, we note that there exists a closed subset $Z'$ 
of the central fiber of $\mathcal{X}$ with $\codim_{\mathcal{X}}(Z')\geq 2$ 
such that $f$ is isomorphism outside $Z'$, since $\mathcal{X}$ is assumed to be 
partially normal. 
Therefore the equality $\DF(\mathcal{B},\mathcal{L}^{\otimes r}(-E))=\DF(\mathcal{X},\mathcal{M}^{\otimes s})$ 
follows}\color{black}{} from the proof of \cite[Proposition 5.1]{RT07}, in particular the equation on each weights $w(-)$ written at the $3$ line above from the end of the proof. 
We note \color{black}{again}\color{black}{} 
that the proof of \cite[Proposition 5.1]{RT07} works essentially without the assumption of normality of $X$.  \end{proof}

\noindent
Proposition \ref{partial normalization} and Proposition \ref{toblup} imply the following Corollary. 
\color{black}{The ``only if" part simply follows the fact that for an arbitrary 
semi test configuration $(\mathcal{Y},\mathcal{N})$, by taking 
$(\Phroj \oplus _{a\geq 0} H^{0}(\mathcal{Y},\mathcal{N}^{\otimes a}), \mathcal{O}(r))$ 
with sufficiently divisible positive integer $r$, we can associate a test configuration 
with the same Donaldson-Futaki invariant as $(\mathcal{Y},\mathcal{N}^{\otimes r})$. 
}\color{black}{}

\begin{Cor}\label{usefulness}
{\rm (i)}
A polarized variety $(X,L)$ is K-semistable if and only if for all 
semi test configurations of the type \ref{DF.formula} $($i.e.\  $(\mathcal{B}=Bl_{\mathcal{J}}(X\times \mathbb{A}^{1}), \mathcal{L}^{\otimes{r}}(-E))$ 
with $\mathcal{B}$ Gorenstein in codimension $1$ $)$, the Donaldson-Futaki invariant is non-negative. 

{\rm (ii)}
A polarized variety $(X,L)$ is K-stable if \color{black}{and only if}\color{black}{} for all semi 
test configurations of the type \ref{DF.formula} $($i.e.\  $(\mathcal{B}=Bl_{\mathcal{J}}(X\times \mathbb{A}^{1}), \mathcal{L}^{\otimes{r}}(-E))$ 
with $\mathcal{B}$ Gorenstein in codimension $1$ $)$, the Donaldson-Futaki invariant is positive. 

\end{Cor}

\noindent
Corollary \ref{usefulness} {\rm (i)} provides further Corollary as follows, 
since the Donaldson-Futaki invariants of the type of \ref{DF.formula} is continuous 
with respect to a variation of $\mathbb{G}_{m}$-linearized polarizations, 
if we extend the framework to \textit{$\mathbb{Q}$-line bundles}. 

\begin{Cor}
K-semistability of $(X,L)$ only depends on $X$ and the numerical equivalent class of $L$. 
\end{Cor}


\section{Some K-(semi)stabilities}\label{app} 

In this section, we give the first direct applications of the formula \ref{DF.formula}. 
That is a concise and algebro-geometric proof 
of some K-(semi)stabilities. 

\begin{Thm}\label{easy.Kstab}

{\rm (i)}
A semi-log-canonical polarized curve $(X,L)$, where $L=\omega_{X}$ 
$($i.e.\ canonically polarized curve$)$ is K-stable. 

{\rm (ii)}
A semi-log-canonical polarized variety $(X,L)$ with numerically trivial 
canonical divisor $K_{X}$ is K-semistable. 

\end{Thm}

\begin{Rem}
Let us recall that a polarized manifold with a constant scalar curvature K\"{a}hler metric is K-polystable, due to the works of \cite{Don05}, \cite{CT08}, \cite{Stp09}, \cite{Mab08b} and \cite{Mab09}. 

Therefore, 
the classical result of the existence of constant curvature metric on an arbitrary compact Riemann surface gives another proof of {\rm (i)} 
 for the case $X$ is smooth over $\mathbb{C}$ as well as 
and the famous result by Yau on the existence of Ricci-flat K\"{a}hler metric on an arbitrary 
polarized Calabi-Yau manifold gives another proof of {\rm (ii)} for the case $X$ is smooth over $\mathbb{C}$. 
\end{Rem}

\begin{proof} 

Due to Corollary \ref{usefulness}, it is sufficient to prove the positivity or non-negativity of the test configurations of 
the form $(\mathcal{B}=Bl_{\mathcal{J}}(X\times \mathbb{A}^{1}), \mathcal{L}^{\otimes{r}}(-E))$ with $\mathcal{B}$ Gorenstein in 
codimension $1$, for which we have 
a formula of Donaldson-Futaki invariants \ref{DF.formula}. 

Let us assume that $X$ is semi-log-canonical, and denotes its normalization 
as $\nu \colon X^{\nu}\rightarrow X$ with its conductor $\cond(\nu)$. 
Then $(X^{\nu} \times \mathbb{A}^{1}, \cond(\nu)\times \mathbb{A}^{1}+X^{\nu} \times \{0\})$ is log-canonical, which can be shown by seeing the discrepancy of the exceptional divisors of the log resolution of $X^{\nu}\times 
\mathbb{A}^{1}$ of the form $\tilde{X}\times \mathbb{A}^{1}\rightarrow X^{\nu}\times \mathbb{A}^{1}$, where 
$\tilde{X}\rightarrow X^{\nu}$ is a log resolution of $(X^{\nu},\cond(\nu))$, 
which exists by \cite{Hir64}. 
This upshot is an easy case of the inversion of adjunction of log-canonicity. 
Now, we want to prove that for an arbitrary (not necessarily closed) point $\eta \in X^{\nu}\times \{0\}$ with $\dim\bar{\{\eta\}}\leq n-1$, 
$\mindiscrep(\eta; (X^{\nu}\times \mathbb{A}^{1},\cond(\nu)\times \mathbb{A}^{1})
\geq 0$, where ``$\mindiscrep$" means the associated minimal discrepancy. We take an exceptional prime divisor 
$E$ above $X^{\nu}\times \mathbb{A}^{1}$ with $\cent_{X^{\nu}\times \mathbb{A}^{1}}(E)=\bar{\{\eta\}}$. Then;  

\begin{equation*}
 \begin{split}
   &  a(E; (X^{\nu}\times \mathbb{A}^{1},\cond(\nu)\times \mathbb{A}^{1}))\\ &= a(E; (X^{\nu}\times \mathbb{A}^{1}, \cond(\nu)\times \mathbb{A}^{1}+X^{\nu}\times \{0\}))+v_{E}(t) \\
                                &\geq \mindiscrep(\eta; (X^{\nu}\times \mathbb{A}^{1}, \cond(\nu)\times \mathbb{A}^{1}+X^{\nu}\times \{0\}))+1, 
 \end{split}                             
\end{equation*} 
where, $v_{E}(-)$ denotes the corresponding discrete valuation for prime divisor $E$. 
Here, $a(-)$ denotes the corresponding discrepancy (cf.\ \cite[Section 2.3]{KM98} or the ``Convention" of this paper). 
Since $(X^{\nu} \times \mathbb{A}^{1}, \cond(\nu)\times \mathbb{A}^{1}+X^{\nu} \times \{0\})$ is log-canonical as we proved, the last line is nonnegative. 

Therefore, we proved that the relative canonical divisor $K_{\mathcal{B}/X\times \mathbb{A}^{1}}$ is effective so that the discrepancy term is nonnegative, if $X$ is semi-log-canonical. 

This ends the proof of $($ii$)$, since the canonical divisor part vanishes in this case. 

For the case $($i$)$, the signature of the canonical divisor part is that of $((\mathcal{L}^{\otimes{r}}-E).(\mathcal{L}^{\otimes{r}}+E))=-(E^{2})$. 
By dividing the flag ideal $\mathcal{J}$ by some power of $t$, 
\color{black}
{
without changing the associated Donaldson-Futaki invariants, 
we can assume $\mathcal{O}/\mathcal{J}'$ is supported in a \textit{proper} 
closed subset of $X\times \{0\}$, not whole of $X\times \{0\}$, 
without loss of generality. 
Consider the normalization $\mu \colon \mathcal{B}^{\mu} \rightarrow \mathcal{B}$. 
We note that there is some connected component $S$ of $\mathcal{B}^{\mu}$, which is a blow up of $0$-dimensional closed subscheme in 
some connected component of $X^{\nu}\times \mathbb{A}^{1}$, 
by the assumption above. Then, we have $(-\mu^{*}E|_{S}^{2})>0$ 
and $(-\mu^{*}E|_{\mathcal{B}^{\mu}\setminus S}^{2})\geq 0$. 
Therefore, we end the proof of {\rm (i)} as well. 
}\color{black}{}

\end{proof}

We end with reviewing that for \textit{asymptotic stability} of these polarized varieties, following is obtained so far by \cite{Mum77}, \cite{Gie82} and \cite{Don01}, in comparison with our results \ref{easy.Kstab}. 

\begin{Thm}\label{res.ast}

{\rm (i)}$($\cite{Mum77}, \cite{Gie82}$)$
A semi-log-canonical polarized curve $(X,L)$, where $L=\omega_{X}$ 
$($i.e.\ canonically polarized curve$)$ is asymptotically stable. 

{\rm (ii)}$($the combination of \cite{Yau78} and \cite{Don01}$)$ 
A smooth polarized manifold $(X,L)$ with numerically trivial canonical divisor $K_{X}$ 
is asymptotically stable. 

\end{Thm}
\noindent
The proof of {\rm (i)} is purely algebro-geometric and by weight's calculation, 
although the proof of {\rm (ii)} is only done by differential geometric methods, which 
depends on the existence of a Ricci-flat K\"ahler metric. Therefore, we need to assume 
that the base field is the complex number field $\mathbb{C}$ for {\rm (ii)}. 

We also note that we can \textit{not} admit semi-log-canonical singularities for Theorem \ref{res.ast} {\rm (ii)}, \textit{nor} can extend {\rm (i)} to higher dimensional varieties 
with semi-log-canonical singularities, as we will show explicit counterexamples in  \cite{Od10}. 


\end{document}